\documentclass[11pt,a4paper]{article}
\usepackage[utf8]{inputenc}

\usepackage{amssymb,amscd,amsfonts,amsbsy,amsmath,amsthm,amsrefs}
\usepackage{dsfont}
\usepackage{enumerate}
\usepackage{mathrsfs}
\usepackage{epsf,epsfig,esint}
\usepackage{pdfsync}
\usepackage[all]{xy}
\usepackage{pdfpages}
\usepackage{hyperref}

\newtheorem{proposition}{Proposition}[section]
\newtheorem{theorem}[proposition]{Theorem}
\newtheorem{lemma}[proposition]{Lemma}
\newtheorem{corollary}[proposition]{Corollary}

\newtheorem*{thm}{Theorem}

\theoremstyle{definition}

\theoremstyle{remark}
\newtheorem{remark}[proposition]{Remark}

\numberwithin{equation}{section}

\providecommand{\keywords}[1]
{
  \small
  \textbf{\textit{Keywords---}} #1
}

\begin{document}

\title{The magnetohydrodynamical system\\ in bounded non smooth domains of dimension $n\ge 3$}
\author{Sylvie Monniaux\footnote{Aix-Marseille Univ., CNRS, I2M UMR7373, Marseille, France 
- {\tt sylvie.monniaux@univ-amu.fr}} \footnote{France-Australia Mathematical Sciences and Interactions ANU - CNRS International Research Laboratory.}\footnote{https://orcid.org/0000-0002-5782-0494}}
\date{}

\maketitle

\abstract{Existence of mild solutions for the magnetohydrodynamical system in ${\mathscr{C}}^1$ domains 
is established in critical spaces in dimension $n\ge 3$. The proof relies on recent regularity results on the Stokes 
operator in ${\mathscr{C}}^1$ domains and a Leibniz-like formula for differential forms proved here in Section\,2.}

\medskip 


{\scriptsize{
\tableofcontents
}}

\section{Introduction}

The magnetohydrodynamical system in a domain $\Omega\subset{\mathds{R}}^3$ on a 
time interval $(0,T)$ ($0<T\le \infty$) as considered in \cite{ST83} (with all constants equal to 1) 
reads
\begin{equation}
\label{mhd}\tag{MHD}
\left\{
\begin{array}{rclcl}
\partial_t u-\Delta u+\nabla \pi+(u\cdot\nabla)u&=&({\rm curl}\,b)\times b
&\mbox{ in }&(0,T)\times\Omega\\
\partial_t b-\Delta b&=&{\rm curl}\,(u\times b)&\mbox{ in }&(0,T)\times\Omega\\
{\rm div}\,u&=&0&\mbox{ in }&(0,T)\times\Omega\\
{\rm div}\,b&=&0&\mbox{ in }&(0,T)\times\Omega\\
\end{array}
\right.
\end{equation}
with the following boundary conditions
\begin{equation}
\label{bc}\tag{BC}
\left\{
\begin{array}{rclcl}
u&=&0&\mbox{ on }&(0,T)\times\partial\Omega\\
\nu\cdot b&=&0&\mbox{ on }&(0,T)\times\partial\Omega\\
\nu\times {\rm curl}\,b&=&0&\mbox{ on }&(0,T)\times\partial\Omega
\end{array}
\right.
\end{equation}
and initial data
\begin{equation}
\label{ic}\tag{ID}
u(0)=u_0, \quad b(0)=b_0
\end{equation}
where $u:(0,T)\times\Omega\to{\mathds{R}}^3$ denotes the {\sl velocity} of the 
(incompressible homogeneous) fluid,
the {\sl magnetic field} (in the absence of magnetic monopole) is denoted by 
$b:(0,T)\times\Omega\to{\mathds{R}}^3$ and $\pi:(0,T)\times\Omega\to{\mathds{R}}$
is the pressure of the fluid. The first equation of \eqref{mhd} corresponds to {\sl Navier-Stokes}
equations subject to
the {\sl Laplace force} $({\rm curl}\,b)\times b$ applied by the magnetic field $b$. 
The second equation of \eqref{mhd} 
describes the evolution of the magnetic field following the so-called {\sl induction} equation.
In the boundary conditions \eqref{bc}, $\nu(x)$ represents the exterior unit normal at a point 
$x\in\partial\Omega$.

For $u$ and $b$ smooth enough with values in ${\mathds{R}}^3$,
we have the following algebraic identities:
\[
({\rm curl}\,b)\times b = -\tfrac{1}{2}\nabla|b|^2 +(b\cdot\nabla)b
\]
and
\[
{\rm curl}\,(u\times b)=(b\cdot\nabla)u-(u\cdot\nabla)b +({\rm div}\,b)\,u -({\rm div}\,u)\,b.
\]
Then \eqref{mhd} in dimension 3 can be equivalently rewritten as
\[
\left\{
\begin{array}{rclcl}
\partial_t u-\Delta u+\nabla \tilde\pi+(u\cdot\nabla)u&=& (b\cdot\nabla)b
&\mbox{ in }&(0,T)\times\Omega\\
\partial_t b-\Delta b&=&(b\cdot\nabla)u-(u\cdot\nabla)b&\mbox{ in }&(0,T)\times\Omega\\
{\rm div}\,u&=&0&\mbox{ in }&(0,T)\times\Omega\\
{\rm div}\,b&=&0&\mbox{ in }&(0,T)\times\Omega\\
\end{array}
\right.
\]
where $\tilde\pi=\pi+\frac{1}{2}\nabla|b|^2$, thanks to the divergence-free condition on $u$ and $b$. 
This formulation led authors (see \cite{Fef14, Ch16, CL18, PT26};
in the latter paper, the authors aknowledge the fact that $b$ should be a 2-form and not a vector field)
to generalise \eqref{mhd} by replacing ${\rm curl}\,b\times b$ by $(b\cdot\nabla)b$ and
${\rm curl}\,(u\times b)$ by $(b\cdot\nabla)u-(u\cdot\nabla)b$ in dimension $n\ge 3$. However,
according to first physical principles, $b$ should be considered as a 2-form which agrees with 
a vector field (via the Hodge-star operator) only in dimension 3. In this paper, we propose a formulation
of the $n$ dimensional MHD system in the language of differential forms (see Subsection~\ref{subsec:diff-form}).
This system has the same scaling invariance as the 3D \eqref{mhd}:
$u_\lambda(t,x)=\lambda u(\lambda^2t,\lambda x)$, 
$b_\lambda(t,x)=\lambda b(\lambda^2t,\lambda x)$ and
$\pi_\lambda(t,x)=\lambda^2 \pi(\lambda^2t,\lambda x)$, $\lambda>0$. This suggests that 
a critical space for $(u,b)$ is ${\mathscr{C}}([0,\infty);L^3({\mathds{R}}^3)^3)\times 
{\mathscr{C}}([0,\infty);L^3({\mathds{R}}^3)^3)$ in dimension 3, or 
${\mathscr{C}}([0,\infty);L^n({\mathds{R}}^n,\Lambda^1)\times 
{\mathscr{C}}([0,\infty);L^n({\mathds{R}}^n,\Lambda^2))$, where $\Lambda^1$ denotes the
space of 1-forms and $\Lambda^2$ the space of 2-forms in dimension $n$. This has also the
advantage to give a direct generalisation to the boundary conditions \eqref{bc} in dimension $n\ge 3$
and agrees with the 2D formulation as in \cite{AG20}.
Moreover, we will see that the condition ${\rm div}\,b=0$ in \eqref{mhd} in dimension 3 is a consequence
of the equations rather than an extra equation of the system; see Lemma~\ref{lem:divb=0}.

The purpose of this paper is to prove existence of solutions of the \eqref{mhd} system in the critical 
space in a bounded ${\mathscr{C}}^1$ domain $\Omega$ in dimension $n\ge 3$.

The first task is to describe a suitable 
version of \eqref{mhd} in dimension $n\ge 3$ with the help of differential forms
and then use the properties of the Dirichlet-Stokes operator in ${\mathscr{C}}^1$ bounded 
domains proved recently in \cite{GS25} and the properties of the so-called Maxwell operator 
established in \cite{MM09a,MM09b,McIM18}. One other important tool to make our strategy
work is a Leibniz-type ``magic formula" (see Theorem~\ref{thm:magic} below), as well as the boundedness
of the Hodge decomposition of $L^p(\Omega,\Lambda)$ for all $1<p<\infty$ 
if $\Omega\subset {\mathds{R}}^n$ is of class ${\mathscr{C}}^1$; see Theorem~\ref{thm:hodge-dec}.
This fact was known for 1-forms (Theorem~11.1 in \cite{FMM98}), but wasn't investigated for the
whole exterior algebra $\Lambda$ as far as the author is aware.

In Section~2, we collect the tools that will be of importance to prove the following
existence result (roughly presented; for a more correct statement, see 
Theorem~\ref{thm:globalexistence} and Theorem~\ref{thm:localexistence}
below):

\begin{thm}
For $u_0,b_0$ in a suitable functional space embedded into $L^n$, there is a global 
mild solution of \eqref{mhd},\eqref{bc} if $u_0,b_0$ are small enough in $L^n$ and a local mild
solution if no size restrictions are imposed on the initial data.
\end{thm}

As usual, the proof of this theorem will use the Picard fixed point theorem for which one of
the main challenging problems is to determine the right space (contained in
${\mathscr{C}}([0,T);L^n)$) for the solutions. 

As a final remark in this introduction, we must point out that if we restrict ourselves to domains
in dimension 3, the same existence results hold in bounded Lipschitz domains with only
a slightly modified proof, using results from \cite{S12}, \cite{T18}, \cite{FMM98}, \cite{McIM18}.

\section{Functional analytic tools}

\subsection{The Dirichlet-Stokes operator in ${\mathscr{C}}^1$ bounded domains}

The Stokes operator with homogeneous Dirichlet boundary conditions has been extensively studied
first in smooth bounded domains by Y.\, Giga (\cite{G81}, Theorem~1 and Theorem~2) and later in 
Lipschitz bounded domains by Z.\,Shen (\cite{S12}, Theorem~1.1) and P.\,Tolksdorf 
(\cite{T18}, Theorem~1.1 and Corollary~1.2). In those references, they prove in the case of smooth
domains that the Stokes semigroup is analytic (with an estimate of the gradient applied to the semigroup)
in $L^p$ for all $1<p<\infty$, and in the case of Lipschitz domains, the analyticity in $L^p$ and the estimate for
the gradient as well is proved for $p$ in an open interval around $2$ containing $[\tfrac{2n}{n+1},\tfrac{2n}{n-1}]$.
This interval in dimension $3$ corresponds to the interval for which the Helmholtz projection is bounded in $L^p$
(projection onto $L^p$ divergence-free vector fields) as proved in \cite{FMM98}, Theorem~11.1.
P.\,Deuring (\cite{D01}, Theorem~1.3) established that there is a bounded Lipschitz domain in ${\mathds{R}}^3$ 
(actually, a bounded domain with a reentrant corner) for which the Dirichlet-Stokes semigroup is not analytic in 
$L^p$ for $p$ close to $1$ or large enough.

Between smooth domains and Lipschitz domains, a class of domains that is worth studying is the class
of bounded domains with ${\mathscr{C}}^1$ boundary. For this class of domains, as mentioned in
\cite{FMM98}, Theorem~11.1, the interval of $p$ for which the Helmholtz projection is bounded in $L^p$
is the whole interval $(1,\infty)$. However, the proof of this fact is much more related to the proof in the
case of Lipschitz domains rather than smooth domains. It is only recently that Z.\,Shen with J.\,Geng
proved that the Stokes resolvent/semigroup is bounded in $L^p$ for all $p\in(1,\infty)$, with an accompanying 
estimate for the gradient of the Stokes resolvent/semigroup in $L^p$ (see \cite{GS25}, Theorem~1.1 for the 
resolvent and Corollary~1.3 for the semigroup). We recall here this result which will be used later in our paper.

\begin{thm}[Theorem~1.1 \& Corollary~1.3 in \cite{GS25}]
Let $\Omega\subset {\mathds{R}}^n$ ($n\ge 3$) be a bounded ${\mathscr{C}}^1$ domain. Then the
Stokes operator with Dirichlet boundary conditions on the space ${\mathbb{L}}^q_\sigma(\Omega)$ defined as
$\{u\in L^q(\Omega;{\mathds{R}}^n) ; {\rm div}\,u=0 \text{ in }\Omega \text{ and } \nu\cdot u=0\text{ on }\partial\Omega\}$
$S_q$ is the negative generator of a uniformly bounded analytic semigroup $(e^{-tS_q})_{t\ge0}$ in
${\mathbb{L}}^q_\sigma(\Omega)$ satisfying:
\begin{equation}
\label{eq:Stokesest}
\sup_{t>0}\bigl(\|e^{-tS_q}\|_{q\to q}+\|\sqrt{t}\nabla e^{-tS_q}\|_{q\to q}\bigr)<\infty.
\end{equation}
\end{thm}

As a corollary, we have the following result.

\begin{corollary}
\label{cor:StokesImprovedRegularity}
Let $1<p<\infty$, $0\le \alpha\le \min\{1,\frac{n}{p}\}$ and define $q\in (1,\infty)$ by
$\frac{1}{q}=\frac{1}{p}-\frac{\alpha}{n}$. The Dirichlet-Stokes semigroup 
$(e^{-tS_p})_{t\ge0}$ on ${\mathbb{L}}_\sigma^p(\Omega)$
satisfies the following estimate
\[
\sup_{t>0}\Bigl(\bigl\|t^{\frac{\alpha}{2}}e^{-tS_p}\bigr\|_{p\to q} +
\bigl\|t^{\frac{1+\alpha}{2}}\nabla e^{-tS_p}\bigr\|_{p\to q} \Bigr)<\infty.
\]
\end{corollary}

\begin{proof}
Note that by the choice of $p,q,\alpha$, we have that $W^{\alpha,p}\hookrightarrow L^q$ in dimension $n$.
Interpolating between the estimates $\|e^{-tS_p}\|_{p\to p}\le C$ and 
$\|\sqrt{t}\nabla e^{-tS_p}\|_{p\to p}\le \tilde{C}$ for all $t>0$ from \eqref{eq:Stokesest}, we have that 
\[
\|e^{-tS_p}\|_{p\to q}\lesssim \|e^{-tS_p}\|_{L^p\to W^{\alpha,p}} 
\lesssim \bigl\|e^{-tS_p}\bigr\|_{p\to p}^{1-\alpha}\bigl\|\nabla e^{-tS_p}\bigr\|_{p\to p}^\alpha 
\le C^{1-\alpha}\bigl(\tfrac{\tilde{C}}{\sqrt{t}}\bigr)^\alpha \lesssim t^{-\frac{\alpha}{2}},
\] 
which gives the first part of the announced estimate. For the second part, it suffices to remark
that for all $t>0$, $e^{-tS_p}=e^{-\frac{t}{2}S_p}e^{-\frac{t}{2}S_p}$ and use the estimate above and
\eqref{eq:Stokesest} as follows
\[
\|\nabla e^{-tS_p}\|_{p\to q}\le \|\nabla e^{-\frac{t}{2}S_q}\|_{q\to q}\|e^{-\frac{t}{2}S_p}\|_{p\to q}
\lesssim \tfrac{1}{\sqrt{t}}\, t^{-\frac{\alpha}{2}} =t^{-\frac{1+\alpha}{2}},
\]
which gives the second part of the estimate of Corollary~\ref{cor:StokesImprovedRegularity}.
\end{proof}

\subsection{Differential forms and Hodge-Laplacians in ${\mathscr{C}}^1$ bounded domains}

A differential form on $\Omega$ is a function from $\Omega$ with values in the exterior algebra
$\Lambda=\Lambda^0\oplus \Lambda ^1\oplus ... \oplus \Lambda^n$ of ${\mathbb{R}}^n$. The 
space of $\ell$-vectors $\Lambda^\ell$ ($1\le \ell\le n$) is the span of $\bigl\{e_J, J\subset\{1,2,...,n\}, |J|=\ell\bigr\}$ where
\[
e_J=e_{j_1}\wedge e_{j_2}\wedge ...\wedge e_{j_\ell}\mbox{ for } J=\{e_{j_1},e_{j_2},..., e_{j_\ell}\}
\mbox{ with }1\le j_1<j_2<...<j_\ell\le n.
\]
The $0$-vectors consist in scalars. That way, a differential form $u:\Omega\to \Lambda$ can be represented by
\[
u=u_0+\sum_{\underset{1\le j_1<j_2<...<j_\ell\le n}{\ell=1}}^nu_{j_1,...,j_\ell}\,e_{j_1}\wedge e_{j_2}\wedge...\wedge e_{j_\ell},
\]
where $u_0$ and $u_{j_1,...,j_\ell}$ for any $1\le \ell\le n$ and any $1\le j_1<j_2<...<j_\ell\le n$ map $\Omega$ to 
${\mathds{R}}$.
Here, $\wedge$ is the exterior product in the exterior algebra $\Lambda$, 
$\lrcorner\,$ is the interior product (or contraction). We denote by $d=\nabla\wedge =\sum_{i=1}^n\partial_i e_i \wedge$ 
the exterior derivative (it satisfies $d^2=0$) and $\delta=-\nabla \lrcorner\,=-\sum_{i=1}^n\partial_i e_i\lrcorner\,$
represents the interior derivative (or co-derivative) acting on differential forms from $\Omega$ to 
the exterior algebra $\Lambda$ ($\delta^2=0$ as well). In dimension $3$, we have these 
correspondences:
\begin{align}
	d :\ & \Lambda^0\sim{\mathds{R}} \xrightarrow[]{\nabla} \Lambda^1\sim{\mathds{R}}^3 \xrightarrow[]{{\rm curl}} 
	\Lambda^2\sim{\mathds{R}}^3 \xrightarrow[]{{\rm div}} \Lambda^3\sim {\mathds{R}}\\
	& \Lambda^0\sim{\mathds{R}} \xleftarrow[-{\rm div}]{} \Lambda^1\sim{\mathds{R}}^3 \xleftarrow[{\rm curl}]{} 
	\Lambda^2\sim{\mathds{R}}^3 \xleftarrow[-\nabla]{} \Lambda^3\sim {\mathds{R}}\ : \ \delta
\end{align}
where it is easy to verify that $d^2=0$ and $\delta^2=0$ using the well-known properties ${\rm curl}\,\nabla=0$
and ${\rm div}\,{\rm curl}=0$. For more on differential forms, we refer to 
\cite{McIM18} \S2.3 or \cite{AMcI04} and the references therein.

We start with establishing Hodge decompositions
in $L^p(\Omega,\Lambda)$ when $\Omega$ is a bounded ${\mathscr{C}}^1$ domain. It has been 
proved in \cite{McIM18}, Section~7 (Theorem~7.1), that the Hodge decompositions
\[
L^p(\Omega,\Lambda)={\sf N}_p(\underline{\delta})\oplus{\sf R}_p(d)
= {\sf N}_p(\delta)\oplus{\sf R}_p(\underline{d})
\]
hold for $\Omega\subset{\mathds{R}}^n$ a bounded Lipschitz domain whenever $p\in(p_H,p^H)$
where
\[
(p^H)'=p_H<\tfrac{2n}{n+1}<\tfrac{2n}{n-1}<p^H.
\]
Here, the operators $\underline{\delta}$ and $\underline{d}$ in $L^p(\Omega,\Lambda)$ are defined as follows:
\begin{align*}
{\sf D}_p(\underline{\delta})&=\{u\in L^p(\Omega,\Lambda); \delta u\in L^p(\Omega,\Lambda)
\mbox{ and }\nu\lrcorner \,u=0 \mbox{ on } \partial\Omega\}, \quad \underline{\delta} u=\delta u ;\\
{\sf D}_p(\underline{d})&=\{u\in L^p(\Omega,\Lambda); d u\in L^p(\Omega,\Lambda)
\mbox{ and }\nu\wedge u=0 \mbox{ on } \partial\Omega\}, \quad \underline{d} u=d u
\end{align*}
and ${\sf N}_p(A)$, ${\sf R}_p(A)$ and ${\sf D}_p(A)$ denote respectively the kernel of the operator A in 
$L^p(\Omega,\Lambda)$, its range in $L^p$ and its domain in $L^p$. The next theorem claims 
that in the case of a ${\mathscr{C}}^1$ bounded domain, we can take $p_H=1$ and $p^H=\infty$.
This has been proved for 1-forms in \cite{FMM98}, Theorem~11.1 and Theorem~11.2. The case of 
$\ell$-forms $2\le \ell\le n-2$ (which is relevant only if $n\ge 4$) is not explicit in the literature as far as the author 
is aware, although all the ingredients of the proof already exist.

\begin{theorem}
\label{thm:hodge-dec}
Let $\Omega\subset{\mathds{R}}^n$ be a bounded ${\mathscr{C}}^1$ domain. Then for all
$p\in(1,\infty)$, the Hodge decompositions
\[
L^p(\Omega,\Lambda)={\sf N}_p(\underline{\delta})\oplus{\sf R}_p(d)
= {\sf N}_p(\delta)\oplus{\sf R}_p(\underline{d})
\]
hold with the boundedness of the accompanying projections.
\end{theorem}

\begin{proof}
The strategy in proving Theorem~7.1 in \cite{McIM18} (the analog of Theorem~\ref{thm:hodge-dec} 
in the case of bounded Lipschitz domains with $p$ in an interval around $2$) relies on the estimate
\begin{equation}
\label{eq:improved-p}
\|u\|_{\frac{pn}{n-1}}\lesssim \|u\|_p+\|\delta u\|_p+\|du\|_p
\end{equation}
for $u\in {\sf D}_p(\underline{\delta})\cap {\sf D}_p(d)$ or $u\in {\sf D}_p(\underline{d})\cap {\sf D}_p(\delta)$.
Here and after, $\|\cdot\|_r$ for $r\in[1,\infty]$ denotes the $L^r$ norm of differential forms defined in $\Omega$ 
with values in $\Lambda$. Our focus now is to prove the estimate \eqref{eq:improved-p} for all $1<p<\infty$ 
if $\Omega$ is a bounded ${\mathscr{C}}^1$ domain,
and then apply the same procedure as in \cite{McIM18} \S7. We focus on the first Hodge decomposition
$L^p(\Omega,\Lambda)={\sf N}_p(\underline{\delta})\oplus{\sf R}_p(d)$, keeping in mind that the proof of
the decomposition $L^p(\Omega,\Lambda)= {\sf N}_p(\delta)\oplus{\sf R}_p(\underline{d})$ follows the same lines. 

Assume that $p\in (1,\infty)$ and $u\in {\sf D}_p(\underline{\delta})\cap {\sf D}_p(d)$. 
That means that $u\in L^p(\Omega,\Lambda)$,
$du\in L^p(\Omega,\Lambda)$, $\delta u\in L^p(\Omega,\Lambda)$ and $\nu\lrcorner\, u=0$ on 
$\partial\Omega$, where $\nu\lrcorner\,u$ is well defined in $B^{-1/p}_p(\partial\Omega,\Lambda)$
since $u,\delta u\in L^p(\Omega,\Lambda)$ as explained in \cite{MMT01}, (11.9) (and (11.10) for 
$\nu\wedge u$ if $u, du\in L^p(\Omega,\Lambda)$). By Theorem~11.2 of \cite{MMT01} and the
remark at the end of \S11, this implies that for $u\in {\sf D}_p(\underline{\delta})\cap {\sf D}_p(d)$,
the quantity $\nu\wedge u$ belongs to  $L^p(\partial\Omega,\Lambda)$. By the closed graph
theorem, we infer that 
\[
\|\nu\wedge u\|_{L^p(\partial\Omega)}\lesssim \|u\|_p+\|du\|_p+\|\delta u\|_p,
\]
where the underlying constant depends on $p$ and $\Omega$. Taking the formula 
$u=\nu\lrcorner\,(\nu\wedge u)+\nu\wedge(\nu\lrcorner\, u)$ on $\partial\Omega$ 
(since the euclidian norm of $\nu(x)$ is equal to $1$ for all $x\in\partial\Omega$) into account,
we obtain that 
\begin{equation}
\label{eq:est-trace}
\|{\rm Tr}_{|_{\partial\Omega}}u\|_{L^p(\partial\Omega)}\lesssim \|u\|_p+\|\delta u\|_p+\|du\|_p
\end{equation}
for all $u\in {\sf D}_p(\underline{\delta})\cap {\sf D}_p(d)$. Since $-\Delta u = d\delta u+\delta d u$, 
we also have that
\begin{equation}
\label{eq:Delta(u)}
\|\Delta u\|_{W^{-1,p}(\Omega,\Lambda)}\le \|du\|_p+\|\delta u\|_p
\end{equation}
From now on, we will apply the following
reasoning to every component of $u\in {\sf D}_p(\underline{\delta})\cap {\sf D}_p(d)$ one by one:
\[
u=u_0+\sum_{\underset{1\le j_1<j_2<...<j_\ell\le n}{\ell=1}}^nu_{j_1,...,j_k}\,e_{j_1}\wedge e_{j_2}\wedge...\wedge e_{j_\ell},
\] 
where $u_0, u_{j_1,...,j_\ell}\in L^p(\Omega)$ with $\Delta u_0, \Delta u_{j_1,...,j_\ell}\in W^{-1,p}(\Omega)$ and
${\rm Tr}_{|_{\partial\Omega}}u_0, {\rm Tr}_{|_{\partial\Omega}}u_{j_1,...,j_\ell} \in L^p(\partial\Omega)$.
To make the next lines easier to read, we will still denote each component by $u$.

By \cite{JK95}, Theorem~5.3, there exists a (unique) harmonic function $v$ in $\Omega$ which 
converges nontangentially to ${\rm Tr}_{|_{\partial\Omega}}u$
almost everywhere in $\partial\Omega$ and satisfies ${\mathscr{N}}(v)\in L^p(\partial\Omega)$,
where ${\mathscr{N}}$ denotes the nontangential maximal function defined by
\[
{\mathscr{N}}(v)(x)=\sup_{y\in\gamma(x)}|v(y)|, \quad x\in\partial\Omega,
\]
$\gamma(x)$ is a nontangential cone with vertex at $x\in\partial\Omega$:
$\gamma(x)=\{y\in\Omega ; |y-x|\le C{\rm dist}(y,\partial\Omega)\}$ for a suitable
constant $C>1$. Now, by Lemma~4.1 in \cite{D03} or Lemma~6.1 in \cite{DM02},
this implies that $v\in L^{\frac{np}{n-1}}(\Omega)$ and
\begin{equation}
\label{eq:est-v}
\|v\|_{\frac{np}{n-1}}\lesssim \|{\mathscr{N}}(v)\|_{L^p(\partial\Omega)}\lesssim\|{\rm Tr}_{|_{\partial\Omega}}u\|_{L^p(\partial\Omega)}.
\end{equation}
Denote now by $w\in W^{1,p}_0(\Omega)$ the unique solution, according to \cite{JK95} Theorem~1.1, of the problem 
$\Delta w=\Delta u \in W^{-1,p}(\Omega)$ with the obvious accompanying estimate. 
Since $W^{1,p}_0(\Omega)\hookrightarrow L^{\frac{np}{n-1}}(\Omega)$ for a bounded domain $\Omega\subset{\mathds{R}}^n$,
we obtain
\begin{equation}
\label{eq:est-w}
\|w\|_{\frac{np}{n-1}}\lesssim \|u\|_{W^{-1,p}(\Omega)}.
\end{equation}
It is now clear that $u=v+w$ and therefore by \eqref{eq:est-v} and \eqref{eq:est-w}, together with \eqref{eq:est-trace} and
\eqref{eq:Delta(u)}, we obtain
\[
\|u\|_{\frac{np}{n-1}}\lesssim \|u\|_p+\|\delta u\|_p+\|du\|_p
\] 
which is the expected estimate \eqref{eq:improved-p}. Next, we proceed as in the proof of Theorem~7.1 in \cite{McIM18}:
we improve the Hodge exponent $p^H = 2+\varepsilon$ by a factor $\frac{n}{n-1}>1$, and this can be done as many
times as wanted. It means that the Hodge decompositions in Theorem~\ref{thm:hodge-dec} hold for any $p\ge 2$ and
by duality, also for all $p\in(1,2]$.
\end{proof}

\begin{corollary}
\label{cor:hodge-dirac-laplacian}
Following the notations of \cite{McIM18} \S5, 
there are two Hodge-Dirac operators in $L^p(\Omega,\Lambda)$: $D_{\|}=d+\underline{\delta}$ with domain 
${\sf D}_p(d)\cap {\sf D}_p(\underline{\delta})$ and $D_{\bot}=\underline{d}+\delta$ with domain
${\sf D}_p(\underline{d})\cap {\sf D}_p(\delta)$. The operators $D_\|$ and $D_\bot$ admit a bounded $S_\mu^o$ 
holomorphic functional calculus in $L^p(\Omega,\Lambda)$ for all $1<p<\infty$ and all $\mu\in(0,\frac{\pi}{2})$.  

Moreover, the Hodge Laplacian $-\Delta_\|:= D_\|^2$ admits a bounded 
$S_{\mu+}^o$ holomorphic functional calculus in $L^p(\Omega,\Lambda)$ and in ${\sf R}_p(\underline{\delta})$
for all $1<p<\infty$ and all $\mu\in(0,\frac{\pi}{2})$. With the same arguments, $-\Delta_\bot:=D_\bot^2$
admits a bounded $S_{\mu+}^o$ holomorphic functional calculus in $L^p(\Omega,\Lambda)$ and in 
${\sf R}_p(\underline{d})$ for all $1<p<\infty$ and all $\mu\in(0,\frac{\pi}{2})$.
\end{corollary}

\begin{proof}
The first part of this corollary is immediate applying \cite{McIM18}, Theorem~5.1, thanks to Theorem~\ref{thm:hodge-dec}.
The second part, concerning the Hodge Laplacians, follows from \cite{McIM18} Corollary~8.1.
\end{proof}

We finish this subsection by giving some extra properties of the Hodge Laplacian when restricted to 
the range of $\underline{d}$ or $\underline{\delta}$. We start by a lemma relying on potential
operators from \cite{CMcI10}, \cite{MMM08} and \cite{McIM18} and methods similar to \cite{MM09b} \S3.

\begin{proposition}
\label{prop:improvedMaxwell}
Let $\Omega\subset{\mathds{R}}^n$ be a bounded ${\mathscr{C}}^1$ domain. Let $1<p<\infty$,
$0\le \alpha<\min\{1,\tfrac{n}{p}\}$ and define $q\in(1,\infty)$ by $\tfrac{1}{q}=\tfrac{1}{p}-\tfrac{\alpha}{n}$.
Then the following two properties hold:
\begin{itemize}
\item[(i)]
if $w\in {\sf D}_p(\delta)\cap {\sf R}_p(\underline{d})$, then $w\in L^q(\Omega,\Lambda)$ and
$\|w\|_q\lesssim \|w\|_p^{1-\alpha}\|\delta w\|_p^\alpha$;
\item[(ii)]
if $w\in {\sf D}_p(d)\cap {\sf R}_p(\underline{\delta})$, then $w\in L^q(\Omega,\Lambda)$ and
$\|w\|_q\lesssim \|w\|_p^{1-\alpha}\|d w\|_p^\alpha$.
\end{itemize}
\end{proposition}

\begin{proof}
Obviously, the two assertions $(i)$ and $(ii)$ are similar and their proofs will follow the same lines. We will focus 
on $(i)$. Using the potential operators from \cite{MMM08} or \cite{CMcI10} with the improvement described
in \cite{McIM18} \S4, we have for $w\in {\sf D}_p(\delta)\cap {\sf R}_p(\underline{d})$:
$w=\delta Q w +Q\delta w +Lw$ where $Q$ maps $L^r(\Omega,\Lambda)$ to $W^{1,r}(\Omega,\Lambda)$
for all $1<r<\infty$, $Q\delta$ is bounded on $L^r(\Omega,\Lambda)$
and $L$ maps $L^r(\Omega,\Lambda)$ to ${\mathscr{C}}^\infty(\overline{\Omega},\Lambda)$. 
Assume for a moment that $w\in {\sf R}_p(\underline{d})\cap L^q(\Omega,\Lambda)={\sf R}_q(\underline{d})$
(see Corollary~4.2 (v) of \cite{McIM18}). We have that $w={\mathbb{P}}_{|_{{\sf R}_q(\underline{d})}}w$ where 
${\mathbb{P}}_{|_{{\sf R}_q(\underline{d})}}$ is the projection from $L^q(\Omega,\Lambda)$ to 
${\sf R}_q(\underline{d})$. By the (second) Hodge decomposition in Theorem~\ref{thm:hodge-dec}, we have
that ${\mathbb{P}}_{|_{{\sf R}_q(\underline{d})}}(\delta Q w)=0$ since $\delta Q w\in {\sf N}_q(\delta)$.
Therefore, $w={\mathbb{P}}_{|_{{\sf R}_q(\underline{d})}}\bigl(Q\delta w +Lw\bigr)$. Now, we have on the one
hand $Q\delta w+Lw\in L^p(\Omega,\Lambda)$ and $\|Q\delta w+Lw\|_p\lesssim \|w\|_p$ and on the other
hand $Q\delta w+Lw\in W^{1,p}(\Omega,\Lambda)$ and $\|Q\delta w+Lw\|_{W^{1,p}}\lesssim \|\delta w\|_p$.
This last estimate is clear from the properties of the operator $Q$. The part where we estimate $Lw$ 
in $L^p$ with $\|\delta w\|_p$ needs some explanation: by the properties of $L$ and the fact that $\Omega$
is bounded, we have directly that $\|Lw\|_{W^{1,p}}\lesssim \|w\|_p$. The fact that the restriction to 
${\sf R}_p(\underline{d})$ of the Hodge-Dirac operator $D_\bot$ (still denoted by $D_\bot$) 
is invertible on ${\sf R}_p(\underline{d})$ allows us to write that $w=D_\bot^{-1}D_\bot w=D_\bot^{-1}\delta w$
and therefore $\|w\|_p=\|D_\bot^{-1}\delta w\|_p\lesssim \|\delta w\|_p$.
Interpolating between these estimates, we obtain
\[
\|Q\delta w+Lw\|_q\lesssim \|Q\delta w+Lw\|_{W^{\alpha,p}}\lesssim \|Q\delta w+Lw\|_{W^{1,p}}^\alpha\|Q\delta w+Lw\|_p^{1-\alpha}
\lesssim \|\delta w\|_p^{\alpha}\|w\|_p^{1-\alpha}
\]
so that $Q\delta w+Lw \in L^q(\Omega,\Lambda)$. The fact that $w={\mathbb{P}}_{|_{{\sf R}_q(\underline{d})}}(w)$
proves $(i)$ for $w\in {\sf R}_q(\underline{d})$ and then, by density of ${\sf R}_q(\underline{d})$ in ${\sf R}_p(\underline{d})$ 
(see Corollary~4.2 (vi) of \cite{McIM18}), we obtain $(i)$ for all $w\in {\sf R}_p(\underline{d})$.
\end{proof}

We obtain then as a consequence the following properties, similar to those stated in \cite{MM09b} Theorem~3.1;
see also \cite{HMM11}, Theorem~1.1.

\begin{corollary}
\label{cor:HodgeLaplacianImprovedRegularity}
Let $\Omega\subset{\mathds{R}}^n$ be a bounded ${\mathscr{C}}^1$ domain. Let $1<p<\infty$,
$0\le \alpha<\min\{1,\tfrac{n}{p}\}$ and define $q\in(1,\infty)$ by $\tfrac{1}{q}=\tfrac{1}{p}-\tfrac{\alpha}{n}$.
Then the following two properties hold:
\begin{itemize}
\item[(i)]
for the Hodge-Laplacian $\Delta_\bot$ on ${\sf R}_p(\underline{d})$:
\[
\sup_{t>0}\Bigl(\bigl\|t^{\frac{\alpha}{2}}e^{t\Delta_\bot}\bigr\|_{{\sf R}_p(\underline{d})\to L^q}+
\bigl\|t^{\frac{1+\alpha}{2}}\delta \,e^{t\Delta_\bot}\bigr\|_{{\sf R}_p(\underline{d})\to L^q}\Bigr)<\infty ;
\]
\item[(ii)]
and for the Hodge-Laplacian $\Delta_{\|}$ on ${\sf R}_p(\underline{\delta})$:
\[
\sup_{t>0}\Bigl(\bigl\|t^{\frac{\alpha}{2}}e^{t\Delta_{\|}}\bigr\|_{{\sf R}_p(\underline{\delta})\to L^q}+
\bigl\|t^{\frac{1+\alpha}{2}}d \,e^{t\Delta_{\|}}\bigr\|_{{\sf R}_p(\underline{\delta})\to L^q}\Bigr)<\infty.
\]
\end{itemize}
\end{corollary}

\subsection{Formulation of \eqref{mhd} in dimension $n\ge 3$}
\label{subsec:diff-form}

In the language of differential forms, with $u=\sum_{i=1}^n u_i e_i$ a $1$-form, 
$b=\sum_{j,k=1}^nb_{jk}(e_j\wedge e_k)$ a $2$-form (since we don't impose $j<k$ in this expression,
it is implicit that $b_{kj}=-b_{jk}$ and that the coefficients are half the ones we would have with the condition
$j<k$) and $\pi$ a 0-form, \eqref{mhd} in dimension $n$ can be formulated as follows:
\begin{equation}
\label{mhd-forms}
\left\{
\begin{array}{rclcl}
\partial_t u-\Delta u+d \pi+(u\cdot\nabla)u&=&\delta\,b\lrcorner\, b
&\mbox{ in }&(0,T)\times\Omega\\
\partial_t b-\Delta b&=&d\,(u\lrcorner\, b)&\mbox{ in }&(0,T)\times\Omega\\
\delta\,u&=&0&\mbox{ in }&(0,T)\times\Omega\\
\end{array}
\right.
\end{equation}
with boundary conditions
\begin{equation}
\label{bc-forms}
\left\{
\begin{array}{rclcl}
u&=&0&\mbox{ on }&(0,T)\times\partial\Omega\\
\nu\wedge b&=&0&\mbox{ on }&(0,T)\times\partial\Omega\\
\nu\wedge \delta b&=&0&\mbox{ on }&(0,T)\times\partial\Omega
\end{array}
\right.
\end{equation}
For more on this, see \cite{D22} or \cite{M21}. 
A remark before going further. In the equations \eqref{mhd-forms}, we do not add the extra condition $d b=0$, 
which would correspond in dimension $3$ to the condition ${\rm div}\,b=0$. Indeed, the fact that $db$ vanishes 
is a consequence of the system \eqref{mhd-forms}-\eqref{bc-forms}.

\begin{lemma}
\label{lem:divb=0}
Let $b$ be a smooth enough $2$-form satisfying $\partial_t b-\Delta b= df$ (for a smooth enough $1$-form $f$)
in $(0,T)\times \Omega$ and $\nu\wedge b=0$, $\nu\wedge \delta b=0$ on $(0,T)\times\partial\Omega$. If $db_0=0$,
then $db=0$ in $(0,T)\times\Omega$.
\end{lemma}

\begin{proof}
Let $b$ be as in the statement of the lemma and let $\beta=db$. Then $\beta$ satisfies $\partial_t\beta-\Delta\beta=0$
in $(0,T)\times \Omega$. As for what happens on the boundary $\partial\Omega$, we have $\nu\wedge \beta =0$ since
$\nu\wedge b=0$ (it is easy to see, by integration by parts, that if $\nu\wedge w=0$, then $\nu\wedge dw=$ for 
$w$ smooth enough). Moreover, $\nu\wedge \delta b=0$ so 
that $\nu\wedge d\delta b=0$ and then 
\[
\nu\wedge \delta \beta =-\nu\wedge \Delta b -\nu\wedge d\delta b=-\nu\wedge \partial_t b=0
\]
since $\nu\wedge b=0$. Therefore, $\beta$ satisfies 
\[
\left\{
\begin{array}{rclcl}
\partial_t\beta-\Delta\beta&=&0&\mbox{ in }&(0,T)\times\Omega\\
\nu\wedge \beta&=&0&\mbox{ on }&(0,T)\times\partial\Omega\\
\nu\wedge \delta\beta&=&0&\mbox{ on }&(0,T)\times\partial\Omega\\
\beta(0)=d b_0 &=&0&\mbox{ in }&\Omega
\end{array}
\right.
\]
which implies that $\beta(t)=e^{t(\Delta_\bot)_{|_{\Lambda^3}}}db_0=0$ for all $t\in[0,T)$. Here, 
$(\Delta_\bot)_{|_{\Lambda^3}}$ is the operator $\Delta_\bot$ defined in Corollary~\ref{cor:hodge-dirac-laplacian}
restricted to differential forms taking values in $\Lambda^3$. 
\end{proof}

\begin{remark}
In the formulation of \eqref{mhd}, \eqref{bc}, \eqref{ic} in 3 dimensions, the equation ${\rm div}\,b=0$ is 
a consequence of the other equations and boundary conditions as soon as the initial data $b_0$ is
divergence-free as well, as shown by the previous lemma.
\end{remark}

\subsection{The magic formula}

In bounded ${\mathscr{C}}^1$ (or Lipschitz) domains, it is known (see \cite{MMT01} \S 11, 
or \cite{C90} in dimension $3$) that the spaces 
$X_N:=\{w\in L^2(\Omega,\Lambda):dw\in L^2(\Omega,\Lambda) \mbox{ and }\delta w\in L^2(\Omega,\Lambda)
\mbox{ with } \nu\wedge w=0 \mbox{ on }\partial\Omega\}$ and
$X_T:=\{w\in L^2(\Omega,\Lambda):dw\in L^2(\Omega,\Lambda) \mbox{ and }\delta w\in L^2(\Omega,\Lambda)
\mbox{ with } \nu\lrcorner \,w=0 \mbox{ on }\partial\Omega\}$
are contained in the Sobolev space $H^{\frac{1}{2}}(\Omega;\Lambda)$. For large classes of more regular 
domains, they are contained in $H^1(\Omega,\Lambda)$ (see for example \cite{ABDG98} in dimension $3$ for 
${\mathscr{C}}^{1,1}$ domains --Theorem~2.9 and Theorem~2.12-- or convex domains --Theorem~2.17-- or
\cite{S95} for smooth domains in all dimensions --Corollary~2.1.6 and Theorem~2.1.7).

This regularity of the spaces $X_T$ and $X_N$ cannot be extended to the case of ${\mathscr{C}}^1$ domains.
In \cite{C19}, M.\,Costabel proved that, in dimension $3$, there exists a ${\mathscr{C}}^1$ bounded domain 
for which $X_T$ and $X_N$ are contained in none of the spaces $H^{\frac{1}{2}+\varepsilon}(\Omega,\Lambda)$
(for all $\varepsilon >0$). 

In the problem that is under consideration here, the right hand-side of the second equation of \eqref{mhd-forms} has 
a term $d(u\lrcorner\, b)$ for $u$ a $1$-form and $b$ a $2$-form. An ideal result would be to estimate
this quantity, say in $L^r$, with the norms of $u$, $\delta u$ and $du$ in $L^p$ and $b$, $\delta b$ and $db$ in 
$L^q$ with $\frac{1}{r}=\frac{1}{p}+\frac{1}{q}$. This is not quite what the following formula will allow, but it 
will be sufficient for our case.

\begin{theorem}[The magic formula]
\label{thm:magic}
Let $u=\sum_{i=1}^nu_ie_i$ be a smooth 1-form and $b=\sum_{j,k=1}^n b_{jk}e_j\wedge e_k$ (where it is implicit
that $b_{kj}=-b_{jk}$) be a smooth 2-form defined on $\Omega\subset{\mathds{R}}^n$. Then the following formula
holds
\begin{equation}
\label{eq:magic}
d(u\lrcorner\, b) +\delta(u\wedge b) = \delta u \wedge b - u\wedge\delta b-u\lrcorner \,db +
(\nabla u+\nabla u^\top)\cdot(b-b^\top).
\end{equation}
\end{theorem} 

\begin{remark}
We also consider $b$ as an antisymmetric matrix $(b_{jk})_{1\le j,k\le n}$. Here, $\nabla u$ denotes
the matrix $(\partial_iu_j)_{1\le i,j\le n}$. In the last term in \eqref{eq:magic}, we identify the 2-form $b$ with the 
$n\times n$ antisymmetric matrix $(b_{jk})_{1\le j,k\le n}$ and $A\cdot B$ denotes the matrix product between the 
two $n\times n$ matrices $A$ and $B$.
\end{remark}

\begin{corollary}[Estimates for the induction equation]
\label{cor:magic-est}
For all bounded domain $\Omega\subset{\mathds{R}}^n$ of class ${\mathscr{C}}^1$, for all $p, \tilde{p},q,\tilde{q}\in(1,\infty)$ with
$\frac{1}{r}:=\frac{1}{\tilde{p}}+\frac{1}{q}=\frac{1}{p}+\frac{1}{\tilde{q}}<1$, there exists a constant $C>0$ such that for all
$u\in L^p(\Omega,\Lambda^1)$ with $\nabla u\in L^{\tilde{p}}(\Omega,{\mathds{R}}^{n\times n})$, ${\rm Tr}_{|_{\partial\Omega}}u=0$
and $b\in L^q(\Omega,\Lambda^2)$ with $db \in L^{\tilde{q}}(\Omega,\Lambda^3)$
and $\delta b\in L^{\tilde{q}}(\Omega,\Lambda^1)$, the following estimate holds
\begin{equation}
\label{eq:magic-est}
\|d(u\lrcorner\, b)\|_r\le C \left(\|\nabla u\|_{\tilde{p}}\|b\|_q+\|u\|_p(\|\delta b\|_{\tilde{q}}+\|db\|_{\tilde{q}})\right).
\end{equation}
\end{corollary}

\begin{proof}[Proof of Theorem~\ref{thm:magic}]
The proof of the formula \eqref{eq:magic} is completely algebraic: it suffices to write every term involved in 
\eqref{eq:magic} and identify their relations. We start with $u=\sum_{i=1}^nu_ie_i$ and $b=\sum_{j,k=1}^n b_{jk}e_j\wedge e_k$
for $u_i, b_{j,k}\in{\mathscr{C}}^\infty(\Omega)$ and compute the terms on the left hand-side of \eqref{eq:magic}. We have
\[
u\lrcorner\, b=\sum_{j,k=1}^nu_k(b_{kj}-b_{jk})e_j
\]
and then
\[
d(u\lrcorner\, b)=\sum_{i,j,k=1}^n\Bigl(\partial_i\bigl(u_k(b_{kj}-b_{jk})\bigr)\Bigr) e_i\wedge e_j
=\nabla u\cdot(b-b^\top) +\sum_{i,j,k=1}^n\bigr(u_k \partial_i(b_{kj}-b_{jk})\bigl)e_i\wedge e_j.
\]
As for the second term on the left hand-side of \eqref{eq:magic}, we have
\begin{align*}
\delta (u\wedge b)&=-\sum_{\ell=1}^n \partial_\ell \Bigl[e_\ell\lrcorner\Bigl(\sum_{i,j,k=1}^nu_ib_{jk}e_i\wedge e_j\wedge e_k\Bigr)\Bigr]
= \sum_{i,j,k=1}^n \bigl(-\partial_k(u_kb_{ij})+\partial_k(u_ib_{kj})-\partial_k(u_ib_{jk})\bigr)e_i\wedge e_j\\
&= \delta u\wedge b +\nabla u^\top\cdot(b-b^\top)-\sum_{i,j,k=1}^n(u_k\partial_kb_{ij}) e_i\wedge e_j
+\sum_{i,j,k=1}^n \bigl(u_i\partial_k(b_{kj}-b_{jk})\bigr)e_i\wedge e_j.
\end{align*}
We next compute the second and third terms on the right hand-side of \eqref{eq:magic}. First, we have 
\[
db = \sum_{i,j,k=1}^n(\partial_ib_{jk})e_i\wedge e_j\wedge e_k
\] 
and then
\[
u\lrcorner \,db =\sum_{i,j,k=1}^n (u_k\partial_k b_{ij})e_i\wedge e_j -\sum_{i,j,k=1}^n \bigl(u_k\partial_i(b_{kj}-b_{jk})\bigr)e_i\wedge e_j.
\]
Second, we have
\[
\delta b = -\sum_{\ell=1}^n\partial_\ell\Bigl[e_\ell\lrcorner\Bigl(\sum_{j,k=1}^n b_{jk} e_j\wedge e_k\Bigr)\Bigr]
= \sum_{j,k=1}^n \partial_k (b_{jk}-b_{kj})e_j
\]
and then
\[
u\wedge \delta b = \sum_{i,j,k=1}^n \bigl(u_i \partial_k (b_{jk}-b_{kj})\bigr)e_i\wedge e_j.
\]
Putting all these terms together, we obtain \eqref{eq:magic}.
\end{proof}

\begin{proof}[Proof of Corollary~\ref{cor:magic-est}]
We have now all tools to prove \eqref{eq:magic-est}. For $\Omega$, $p,\tilde{p},q, \tilde{q},r$ as in Corollary~\ref{cor:magic-est}, let
$u\in {\mathscr{C}}_c^\infty(\Omega,\Lambda^1)$ and $b\in {\mathscr{C}}^\infty(\overline{\Omega},\Lambda^2)$. 
First, note that in that case $u\lrcorner\,b\in {\mathscr{C}}^\infty(\overline{\Omega},\Lambda^1) \subset L^r(\Omega,\Lambda^1)$, 
$d(u\lrcorner\,b) \in {\mathscr{C}}^\infty(\overline{\Omega},\Lambda^2) \subset L^r(\Omega,\Lambda^2)$
and $\nu \wedge (u\lrcorner\,b) =0$ on $\partial\Omega$. This means that 
$d(u\lrcorner\,b) \in {\sf R}_r(\underline{d})_{|_{\Lambda^2}}$ and therefore
\[
{\mathbb{P}}_{{\sf R}_r(\underline{d})_{|_{\Lambda^2}}}\bigl(d(u\lrcorner\,b)\bigr)=d(u\lrcorner\,b).
\]
It is also clear that $\delta (u\wedge b)\in {\sf N}_r(\delta)_{|_{\Lambda^2}}$ so that, thanks to the (second) Hodge
decomposition in Theorem~\ref{thm:hodge-dec},
\[
{\mathbb{P}}_{{\sf R}_r(\underline{d})_{|_{\Lambda^2}}}\bigl(\delta(u\wedge b)\bigr)=0.
\]
Now, using the magic formula \eqref{eq:magic}, we have that
\[
d(u\lrcorner\,b) = {\mathbb{P}}_{{\sf R}_r(\underline{d})_{|_{\Lambda^2}}}
\bigl(\delta u \wedge b - u\wedge\delta b-u\lrcorner \,db + (\nabla u+\nabla u^\top)\cdot(b-b^\top)\bigr).
\]
Taking the $L^r(\Omega,\Lambda^2)$-norm of this last expression, we obtain
\[
\|d(u\lrcorner\, b)\|_r\lesssim \bigl(\|\delta u\|_{\tilde{p}}+\|du\|_{\tilde{p}}+\|\nabla u\|_{\tilde{p}}\bigr)\|b\|_q
+\|u\|_p\bigl(\|db\|_{\tilde{q}}+\|\delta b\|_{\tilde{q}}\bigr)
\]
which gives \eqref{eq:magic-est} since both $\delta u$ and $du$ can be estimated by $\nabla u$ in $L^{\tilde{p}}$. We conclude the proof
by saying that ${\mathscr{C}}_c^\infty(\Omega,\Lambda^1)$ is dense in 
$\bigl\{u\in L^p(\Omega,\Lambda^1) ; \nabla u\in L^{\tilde{p}}(\Omega,{\mathds{R}}^{n\times n}) \mbox{ and }{\rm Tr}_{|_{\partial\Omega}}u=0\bigr\}$
and ${\mathscr{C}}^\infty(\overline{\Omega},\Lambda^2)$ is dense in the space
$\bigl\{b\in L^q(\Omega,\Lambda^2); db\in L^{\tilde{q}}(\Omega,\Lambda^3)\mbox{ and } \delta b\in L^{\tilde{q}}(\Omega,\Lambda^1)\bigr\}$.
\end{proof}

\section{Existence results in critical spaces}

\subsection{Mild solutions}

As explained in the Introduction, the critical spaces for $u$ and $b$ in the system \eqref{mhd} in dimension~$n$ are
${\mathscr{C}}([0,T);{\mathbb{L}}_\sigma^n(\Omega)\times {\sf R}_n(\underline{d})_{|_{\Lambda^2}})$ where
${\mathbb{L}}^n_\sigma(\Omega)={\sf N}_n(\underline{\delta})_{|_{\Lambda^1}}$ and ${\sf R}_n(\underline{d})_{|_{\Lambda^2}}$
are respectively closed subspaces of $L^n(\Omega,\Lambda^1)$ and $L^n(\Omega,\Lambda^2)$ and $0<T\le \infty$. We say that 
$(u,b)\in {\mathscr{C}}([0,T);{\mathbb{L}}_\sigma^n(\Omega)\times {\sf R}_n(\underline{d})_{|_{\Lambda^2}})$ is a {\tt mild solution} 
of \eqref{mhd-forms}-\eqref{bc-forms} with initial conditions 
$(u_0,b_0)\in {\mathbb{L}}_\sigma^n(\Omega)\times {\sf R}_n(\underline{d})_{|_{\Lambda^2}}$ if $(u,b)$ satisfies
the equations in the following integral form
\begin{align}
\label{eq:mild_u}
u&={\mathscr{S}}_nu_0-{\mathscr{S}}_r*{\mathbb{P}}_r\bigl((u\cdot\nabla)u\bigr)
+{\mathscr{S}}_r*{\mathbb{P}}_r\bigl(\delta b\lrcorner\,b\bigr)\\
\nonumber
&={\frak{u}}_0 +B_1(u,u)+B_2(b,b),\\
\label{eq:mild_b}
b&={\mathscr{M}}_nb_0 +{\mathscr{M}}_r*{\mathbb{Q}}_r\bigl(d(u\lrcorner\,b)\bigr)\\
\nonumber
&={\frak{b}}_0+B_3(u,b),
\end{align}
where ${\mathbb{P}}_r:={\mathbb{P}}_{{\sf N}_r(\underline{\delta})_{|_{\Lambda^1}}}$ denotes the projection from
$L^r(\Omega,\Lambda^1)$ to ${\sf N}_r(\underline{\delta})_{|_{\Lambda^1}}$ (for a suitable $r\in(1,\infty)$),
${\mathbb{Q}}_r:={\mathbb{P}}_{{\sf R}_r(\underline{d})_{|_{\Lambda^2}}}$ denotes the projection from
$L^r(\Omega,\Lambda^2)$ to ${\sf R}_r(\underline{d})_{|_{\Lambda^2}}$,
${\mathscr{S}}_q(t)=e^{-tS_q}$ ($q=n$ or $r$, $S_q$ being the Dirichlet-Stokes operator of \S\,2.1), 
${\mathscr{M}}_q(t)=e^{-tM_q}$ ($M_q$ is the negative Hodge Laplacian $-{\Delta_\bot}$ defined in \S\,2.2
restricted to ${\sf R}_q(\underline{d})_{|_{\Lambda^2}}$ with $q=n$ or $r$). 
The sign $*$ denotes the time convolution on $[0,\infty)$.

The classical strategy to prove existence of a mild solution of \eqref{mhd-forms}-\eqref{bc-forms} is via a fixed point 
argument in a well chosen Banach space, as done by Fujita and Kato in \cite{FK64} for the Navier-Stokes system. 

For $0<T\le \infty$, define
\begin{align*}
{\mathscr{X}}_T:=\bigl\{&(u,b)\in {\mathscr{C}}((0,T);{\mathbb{L}}_\sigma^{2n}(\Omega)\times {\sf R}_{2n}(\underline{d})_{|_{\Lambda^2}});
t\mapsto t^{\frac{1}{4}}u(t)\in {\mathscr{C}}([0,T);{\mathbb{L}}_\sigma^{2n}(\Omega)), \\
&t\mapsto \sqrt{t}\, \nabla u(t) \in {\mathscr{C}}([0,T);L^n(\Omega,{\mathds{R}}^{n\times n})), 
t\mapsto t^{\frac{1}{4}}b(t)\in {\mathscr{C}}([0,T);{\sf R}_{2n}(\underline{d})_{|_{\Lambda^2}})\\
&\mbox{ and }t\mapsto \sqrt{t}\, \delta b(t) \in {\mathscr{C}}([0,T);L^n(\Omega,\Lambda^1)\bigr\}\\
=&{\mathscr{U}}_T\times {\mathscr{B}}_T,
\end{align*}
with norm $\bigl\|(u,b)\bigr\|_{{\mathscr{X}}_T}=\bigl\|u\bigr\|_{{\mathscr{U}}_T}+\bigl\|b\bigr\|_{{\mathscr{B}}_T}$,
where 
\begin{align*}
{\mathscr{U}}_T:=\bigl\{&u\in {\mathscr{C}}((0,T);{\mathbb{L}}_\sigma^{2n}(\Omega)); 
t\mapsto t^{\frac{1}{4}}u(t)\in {\mathscr{C}}([0,T);{\mathbb{L}}_\sigma^{2n}(\Omega)) \\
&\mbox{and } t\mapsto \sqrt{t}\, \nabla u(t) \in {\mathscr{C}}([0,T);L^n(\Omega,{\mathds{R}}^{n\times n}))\bigr\}
\end{align*}
endowed with the norm 
\[
\bigl\|u\bigr\|_{{\mathscr{U}}_T}:=
\sup_{0< t <T}\bigl\|t^{\frac{1}{4}}u(t)\bigr\|_{2n}+\sup_{0<t<T}\bigl\|\sqrt{t}\,\nabla u(t)\bigr\|_n,
\]
and
\begin{align*}
{\mathscr{B}}_T:=\bigl\{&b\in {\mathscr{C}}((0,T);{\sf R}_{2n}(\underline{d})_{|_{\Lambda^2}}) ;
t\mapsto t^{\frac{1}{4}}b(t)\in {\mathscr{C}}([0,T);{\sf R}_{2n}(\underline{d})_{|_{\Lambda^2}})\\
&\mbox{and }t\mapsto \sqrt{t}\, \delta b(t) \in {\mathscr{C}}([0,T);L^n(\Omega,\Lambda^1))\bigr\}
\end{align*}
endowed with the norm
\[
\bigl\|b\bigr\|_{{\mathscr{B}}_T}:=
\sup_{0< t <T}\bigl\|t^{\frac{1}{4}}b(t)\bigr\|_{2n}+\sup_{0<t<T}\bigl\|\sqrt{t}\,\delta b(t)\bigr\|_n.
\]
We state here a collection of lemmas which proofs we defer to the Appendix. The results are classical, but need to be
tailored to our framework.

\begin{lemma}
\label{lem:initialconditions}
For all $(u_0,b_0)\in {\mathbb{L}}_\sigma^n(\Omega)\times {\sf R}_n(\underline{d})_{|_{\Lambda^2}}$, we
have that $({\frak{u}}_0, {\frak{b}}_0)\in {\mathscr{X}}_T$ for all $T\in(0,\infty]$. Moreover, there exists a constant
$\gamma>0$ (independent of $T>0$) such that
\[
\bigl\|({\frak{u}}_0, {\frak{b}}_0)\bigr\|_{{\mathscr{X}}_T}\le \gamma (\|u_0\|_n+\|b_0\|_n).
\]
We also have that
\[
\bigl\|({\frak{u}}_0, {\frak{b}}_0)\bigr\|_{{\mathscr{X}}_T}\xrightarrow[T\to 0]{}0.
\]
\end{lemma}

\begin{lemma}
\label{lem:B1}
For all $T>0$, the bilinear operator $B_1$ defined for $u_1,u_2\in{\mathscr{U}}_T$ by 
$B_1(u_1,u_2):= -{\mathscr{S}}_n*{\mathbb{P}}_{\frac{2n}{3}}\bigl((u_1\cdot\nabla)u_2\bigr)$
is bounded from ${\mathscr{U}}_T\times {\mathscr{U}}_T$ to ${\mathscr{U}}_T$ for all $T\in(0,\infty]$.
Moreover, for all $u\in {\mathscr{U}}_T$, 
\[
B_1(u,u)\in {\mathscr{C}}([0,T);{\mathbb{L}}_\sigma^n(\Omega)).
\]
\end{lemma}

\begin{lemma}
\label{lem:B2}
For all $T>0$, the bilinear operator $B_2$ defined by 
$B_2(b_1,b_2):= {\mathscr{S}}_n*{\mathbb{P}}_{\frac{2n}{3}}\bigl((b_1\lrcorner\,u_2\bigr)$
is bounded from ${\mathscr{B}}_T\times {\mathscr{B}}_T$ to ${\mathscr{U}}_T$ for all $T\in(0,\infty]$. 
Moreover, for all $b \in {\mathscr{B}}_T$, 
\[
B_2(b,b)\in {\mathscr{C}}([0,T);{\mathbb{L}}_\sigma^n(\Omega)).
\]
\end{lemma}

\begin{lemma}
\label{lem:B3}
For all $T>0$, the bilinear operator $B_3$ defined by 
$B_3(u,b):= {\mathscr{M}}_{\frac{2n}{3}}*{\mathbb{Q}}_{\frac{2n}{3}}\bigl(d(u\lrcorner\,b)\bigr)$
is bounded from ${\mathscr{U}}_T\times {\mathscr{B}}_T$ to ${\mathscr{B}}_T$ for all $T\in(0,\infty]$.
Moreover, for all $(u,b)\in {\mathscr{X}}_T$, 
\[
B_3(u,b) \in {\mathscr{C}}([0,T);{\sf R}_n(\underline{d})_{|_{\Lambda^2}}).
\]
\end{lemma}

\subsection{Global existence}

We are now in position to state and prove a more precise version of the global existence result announced in the Introduction.

\begin{theorem}
\label{thm:globalexistence}
Let $\Omega\subset{\mathds{R}}^n$ be a bounded ${\mathscr{C}}^1$ domain and let $T\in(0,\infty]$.
There exists $\varepsilon>0$ such that for all $(u_0,b_0)\in {\mathbb{L}}_\sigma^n(\Omega)\times {\sf R}_n(\underline{d})_{|_{\Lambda^2}}$
satisfying 
$\|u_0\|_n+\|b_0\|_n\le \varepsilon$,
there exists a mild solution $(u,b)\in {\mathscr{C}}([0,T);{\mathbb{L}}_\sigma^n(\Omega)\times {\sf R}_n(\underline{d})_{|_{\Lambda^2}})$
of \eqref{mhd-forms}-\eqref{bc-forms} with initial conditions $(u_0,b_0)$.
\end{theorem}

\begin{proof}
We rewrite the problem of finding mild solutions of \eqref{mhd-forms}-\eqref{bc-forms} on $[0,T)$ as follows:
$U=(u,b)\in {\mathscr{X}}_T$ satisfies $U= {\frak{U}}_0+\Phi(U,U)$ where 
${\frak{U}}_0=({\frak{u}}_0,{\frak{b}}_0)$ and $\Phi:{\mathscr{X}}_T\times{\mathscr{X}}_T\to {\mathscr{X}}_T$
is defined by $\Phi(U_1,U_2)=
\left(B_1(u_1,u_2)+B_2(b_1,b_2) ,B_3(u_1,b_2)\right)$.
Picard's fixed point theorem asserts that such a $U$ exists if 
$\bigl\|{\frak{U}}_0\bigr\|_{{\mathscr{X}}_T}\le \epsilon <\frac{1}{4C_T}$ with 
$C_T=\bigl\|\Phi\bigr\|_{{\mathscr{X}}_T\times {\mathscr{X}}_T \to {\mathscr{X}}_T}$: this is granted by 
Lemmas~\ref{lem:B1}, \ref{lem:B2}, \ref{lem:B3}. It suffices to choose $\varepsilon=\frac{\epsilon}{\gamma}$ 
(where $\gamma$ is the constant appearing in Lemma~\ref{lem:initialconditions}) such that 
$\|u_0\|_n+\|b_0\|_n\le \varepsilon$ implies $\bigl\|{\frak{U}}_0\bigr\|_{{\mathscr{X}}_T}\le \epsilon$. 
Now, thanks to Lemmas~\ref{lem:B1}, \ref{lem:B2}, \ref{lem:B3}, we know that $\Phi(U,U)$ belongs to  
${\mathscr{C}}([0,T);{\mathbb{L}}_\sigma^n(\Omega)\times {\sf R}_n(\underline{d})_{|_{\Lambda^2}})$. It remains to show that 
${\frak{U}}_0\in{\mathscr{C}}([0,T);{\mathbb{L}}_\sigma^n(\Omega)\times {\sf R}_n(\underline{d})_{|_{\Lambda^2}})$.
This is immediate taking into account the fact that ${\mathscr{S}}_n$ and ${\mathscr{M}}_n$ are strongly continuous semigroups
on ${\mathbb{L}}_\sigma^n(\Omega)$ for ${\mathscr{S}}$ (see Theorem~1.1 in \cite{GS25})
and on ${\sf R}_n(\underline{d})_{|_{\Lambda^2}}$ for ${\mathscr{M}}$ (see Corollary~\ref{cor:hodge-dirac-laplacian}).
\end{proof}

\subsection{Local existence}

The local existence result reads as follows.

\begin{theorem}
\label{thm:localexistence}
Let $\Omega\subset{\mathds{R}}^n$ be a bounded ${\mathscr{C}}^1$ domain and let 
$(u_0,b_0)\in {\mathbb{L}}_\sigma^n(\Omega)\times {\sf R}_n(\underline{d})_{|_{\Lambda^2}}$.
Then there exists $T>0$ depending on $(u_0,b_0)$ such that \eqref{mhd-forms}-\eqref{bc-forms} with initial conditions $(u_0,b_0)$
admits a mild solution $(u,b)\in {\mathscr{C}}([0,T);{\mathbb{L}}_\sigma^n(\Omega)\times {\sf R}_n(\underline{d})_{|_{\Lambda^2}})$.
\end{theorem}

\begin{proof}
The proof follows the lines of the proof of global existence. With the same notations as in the proof of 
Theorem~\ref{thm:globalexistence}, we can find $T>0$ such that 
$\bigl\|{\frak{U}}_0\bigl\|_{{\mathscr{X}}_T}\le \varepsilon$ for $0<\varepsilon<\frac{1}{4C_T}$.
This is possible since $C_T=\bigl\|\Phi\bigr\|_{{\mathscr{X}}_T\times {\mathscr{X}}_T \to {\mathscr{X}}_T}$ does not depend on $T>0$
(see the proofs of Lemmas~\ref{lem:B1}, \ref{lem:B2}, \ref{lem:B3})
and $\bigl\|{\frak{U}}_0\bigr\|_{{\mathscr{X}}_T}\xrightarrow[T\to 0]{}0$, as proved in Lemma~\ref{lem:initialconditions}.
\end{proof}

\appendix

\section{Deferred proofs}

\subsection{About initial conditions}

\begin{proof}[Proof of Lemma~\ref{lem:initialconditions}]
The fact that $\frak{u}_0$ belongs to ${\mathscr{U}}_T$ comes from the properties of the Dirichlet-Stokes
semigroup stated in Corollary~\ref{cor:StokesImprovedRegularity} with $p=n$, $\alpha=\frac{1}{2}$ and 
$q=2n$ for the part $\bigl\|t^{\frac{1}{4}}\frak{u}_0(t)\bigr\|_{2n}$ in the definition of ${\mathscr{U}}_T$ and from \eqref{eq:Stokesest}
for the part $\bigl\|\sqrt{t}\,\nabla \frak{u}_0(t)\bigr\|_n$ in the definition of ${\mathscr{U}}_T$. We have for all $T>0$
\begin{align*}
\bigl\|\frak{u}_0\bigr\|_{{\mathscr{U}}_T} =& \sup_{0<t<T}\Bigl(\bigl\|t^{\frac{1}{4}}\frak{u}_0(t)\bigr\|_{2n}
+\bigl\|\sqrt{t}\,\nabla\frak{u}_0(t)\bigr\|_n\Bigr)\\
\lesssim &\Bigl(\sup_{t>0}\bigl\|t^{\frac{1}{4}}e^{-tS_n}\bigr\|_{n\to 2n}
+\sup_{t>0}\bigl\|\sqrt{t}\,\nabla e^{-tS_n}\bigr\|_{n\to n}\Bigr)\|u_0\|_n \lesssim \|u_0\|_n,
\end{align*}
where the constants involved do not depend on $T>0$. To prove that $\frak{b}_0$ belongs to ${\mathscr{B}}_T$,
we follow the same lines, using the properties of the semigroup ${\mathscr{M}}_n$ inherited from the properties
of the Hodge-Laplacian proved in Corollary~\ref{cor:HodgeLaplacianImprovedRegularity} $(ii)$. We have for all $T>0$
\begin{align*}
\|\frak{b}_0\|_{{\mathscr{B}}_T} =& 
\sup_{0<t<T}\Bigl(\bigl\|t^{\frac{1}{4}}\frak{b}_0(t)\bigr\|_{2n}+\bigl\|\sqrt{t}\,\delta \frak{b}_0(t)\bigr\|_n\Bigr)\\
\lesssim & \Bigl(\sup_{t>0}\bigl\|t^{\frac{1}{4}}e^{-tM_n}\bigr\|_{{\sf R}_n(\underline{d})_{|_{\Lambda^2}}\to L^{2n}(\Omega,\Lambda^2)}
+\sup_{t>0}\bigl\|\sqrt{t}\,\delta e^{-tM_n}\bigr\|_{{\sf R}_n(\underline{d})_{|_{\Lambda^2}}\to L^n(\Omega,\Lambda^1)}\Bigr)\|b_0\|_n\\
\lesssim &\|b_0\|_n,
\end{align*}
where the constants involved do not depend on $T>0$. This gives the existence of a constant $\gamma>0$ independent of $T>0$
such that 
\[
\bigl\|({\frak{u}}_0, {\frak{b}}_0)\bigr\|_{{\mathscr{X}}_T}\le \gamma (\|u_0\|_n+\|b_0\|_n).
\]
It remains to show that $\bigl\|({\frak{u}}_0, {\frak{b}}_0)\bigr\|_{{\mathscr{X}}_T}\xrightarrow[T\to 0]{}0$. Let 
$(u_0,b_0)\in {\mathbb{L}}_\sigma^n(\Omega)\times {\sf R}_n(\underline{d})_{|_{\Lambda^2}}$. Let $\varepsilon>0$
and choose $u_{0,\varepsilon} \in W^{1,n}(\Omega,\Lambda^1)\cap{\mathbb{L}}_\sigma^n(\Omega)$ such
that $\|u_0 -u_{0,\varepsilon}\|_n\le \frac{\varepsilon}{4\gamma}$. Choose also 
$b_{0,\varepsilon}\in {\sf R}_{2n}(\underline{d})_{|_{\Lambda^2}}\cap {\sf D}_n(\delta)_{|_{\Lambda^2}}$ such
that $\|b_0 -b_{0,\varepsilon}\|_n\le \frac{\varepsilon}{4\gamma}$. By the previous estimate, we have that 
\[
\bigl\|\bigl({\mathscr{S}}_n(u_0 -u_{0,\varepsilon}), {\mathscr{M}}_n(b_0 -b_{0,\varepsilon})\bigr)\bigr\|_{{\mathscr{X}}_T}
\le \tfrac{\varepsilon}{2}.
\]
Now, by the properties of the Dirichlet-Stokes semigroup ${\mathscr{S}}$, we have
\[
\sup_{0<t<T}\|t^{\frac{1}{4}}e^{-tS_n}u_{0,\varepsilon}\|_{2n}\le T^{\frac{1}{4}}\|u_{0,\varepsilon}\|_{2n}
\le T^{\frac{1}{4}}\|u_{0,\varepsilon}\|_n^{\frac{1}{2}}\|\nabla u_{0,\varepsilon}\|_n^{\frac{1}{2}}\xrightarrow[T\to\infty]{}0
\]
and 
\[
\sup_{0<t<T}\|t^{\frac{1}{2}}\nabla e^{-tS_n}u_{0,\varepsilon}\|_n\lesssim T^{\frac{1}{2}}\|u_{0,\varepsilon}\|_{W^{1,n}}
\xrightarrow[T\to\infty]{}0.
\]
The properties of the semigroup ${\mathscr{M}}$ and Proposition~\ref{prop:improvedMaxwell}
entail the following estimates 
\[
\sup_{0<t<T}\|t^{\frac{1}{4}}e^{-tM_n}b_{0,\varepsilon}\|_{2n}\le T^{\frac{1}{4}}\|b_{0,\varepsilon}\|_{2n}
\le T^{\frac{1}{4}}\|b_{0,\varepsilon}\|_n^{\frac{1}{2}}\|\delta b_{0,\varepsilon}\|_n^{\frac{1}{2}}\xrightarrow[T\to\infty]{}0
\]
and 
\[
\sup_{0<t<T}\|t^{\frac{1}{2}}\delta e^{-tM_n}u_{0,\varepsilon}\|_n\lesssim T^{\frac{1}{2}}\|\delta b_{0,\varepsilon}\|_n
\xrightarrow[T\to\infty]{}0.
\]
We then choose $T>0$ small enough so that 
\[
\|{\mathscr{S}}_nu_{0,\varepsilon}\|_{{\mathscr{U}}_T}
=\sup_{0<t<T}\|t^{\frac{1}{4}}e^{-tS_n}u_{0,\varepsilon}\|_{2n}+\sup_{0<t<T}\|t^{\frac{1}{2}}\nabla e^{-tS_n}u_{0,\varepsilon}\|_n
\le \tfrac{\varepsilon}{4}
\]
and 
\[
\|{\mathscr{M}}_nb_{0,\varepsilon}\|_{{\mathscr{B}}_T}
\sup_{0<t<T}\|t^{\frac{1}{4}}e^{-tM_n}b_{0,\varepsilon}\|_{2n}+ \sup_{0<t<T}\|t^{\frac{1}{2}}\delta e^{-tM_n}u_{0,\varepsilon}\|_n
\le \tfrac{\varepsilon}{4}.
\]
Ultimately, this gives
$\bigl\|(\frak{u}_0,\frak{b}_0)\bigr\|_{{\mathscr{X}}_T}\le \varepsilon$. Since $\varepsilon>0$ was arbitrary, we proved that 
$\bigl\|({\frak{u}}_0, {\frak{b}}_0)\bigr\|_{{\mathscr{X}}_T}$ goes to $0$ as $T$ goes to $0$, as claimed.
\end{proof}

\subsection{The $B$ operators}

\begin{proof}[Proof of Lemma~\ref{lem:B1}]
Assume that $u_1,u_2\in{\mathscr{U}}_T$. Then we have that 
$s\mapsto \bigl[(u_1(s)\cdot\nabla)u_2\bigr](s) \in {\mathscr{C}}((0,T);L^{\frac{2n}{3}}(\Omega,{\mathds{R}}^n))$
with 
\[
\bigl\|s^{\frac{3}{4}}\bigl[(u_1(s)\cdot\nabla)u_2\bigr](s)\bigr\|_{\frac{2n}{3}}
\le \bigl\|s^{\frac{1}{4}}u_1(s)\bigr\|_{2n}\bigl\|\sqrt{s}\,\nabla u_2(s)\bigr\|_n
\le \bigl\|u_1\bigr\|_{{\mathscr{U}}_T}\bigl\|u_2\bigr\|_{{\mathscr{U}}_T}
\]
since $L^{2n} \cdot L^n\hookrightarrow L^{\frac{2n}{3}}$. Moreover, $W^{1,\frac{2n}{3}}\hookrightarrow L^{2n}$ in dimension $n$: by the 
properties of the Stokes semigroup ${\mathscr{S}}$ described in \eqref{eq:Stokesest}, we obtain that
\[
\bigl\|-{\mathscr{S}}_n*{\mathbb{P}}_{\frac{2n}{3}}\bigl((u_1\cdot\nabla)u_2\bigr)(t)\bigr\|_{2n}
\lesssim \Bigl(\int_0^t\frac{1}{(t-s)^{\frac{1}{2}}}\frac{1}{s^{\frac{3}{4}}}\, {\rm d}s\Bigr) 
\bigl\|u_1\bigr\|_{{\mathscr{U}}_T}\bigl\|u_2\bigr\|_{{\mathscr{U}}_T}
= c_1 t^{-\frac{1}{4}}\bigl\|u_1\bigr\|_{{\mathscr{U}}_T}\bigl\|u_2\bigr\|_{{\mathscr{U}}_T},
\]
for all $t\in (0,T)$.
By Corollary~\ref{cor:StokesImprovedRegularity} with $p=\frac{2n}{3}$, $\alpha=\frac{1}{2}$ and $q=n$, the same reasoning as before
allows us to estimate $B_1(u_1,u_2)$ as follows:
\[
\bigl\|-\nabla\bigl[{\mathscr{S}}_n*{\mathbb{P}}_{\frac{2n}{3}}\bigl((u_1\cdot\nabla)u_2\bigr)\bigr](t)\bigr\|_n
\lesssim \Bigl(\int_0^t\frac{1}{(t-s)^{\frac{3}{4}}}\frac{1}{s^{\frac{3}{4}}}\, {\rm d}s\Bigr) 
\bigl\|u_1\bigr\|_{{\mathscr{U}}_T}\bigl\|u_2\bigr\|_{{\mathscr{U}}_T}
= \tilde{c}_1 t^{-\frac{1}{2}}\bigl\|u_1\bigr\|_{{\mathscr{U}}_T}\bigl\|u_2\bigr\|_{{\mathscr{U}}_T},
\]
for all $t\in (0,T)$. With these two estimates in hand, we obtain that
$\bigl\|B_1(u_1,u_2)\bigr\|_{{\mathscr{U}}_T}\lesssim \bigl\|u_1\bigr\|_{{\mathscr{U}}_T}\bigl\|u_2\bigr\|_{{\mathscr{U}}_T}$ 
where the constant in the estimate depends only on the norm of ${\mathbb{P}}_{\frac{2n}{3}}$, the quantity in \eqref{eq:Stokesest}, 
$c_1$ and $\tilde{c}_1$, so does not depend on $T>0$.
It remains to prove continuity on $[0,T)$ in ${\mathbb{L}}_\sigma^n(\Omega)$. Thanks to the continuity of $(u_1\cdot\nabla)u_2$ in 
$L^{\frac{2n}{3}}(\Omega,{\mathds{R}}^n))$ and the strong continuity of the Stokes semigroup, the convolution
$-{\mathscr{S}}_n*{\mathbb{P}}_{\frac{2n}{3}}\bigl((u_1\cdot\nabla)u_2\bigr)$ is continuous on $(0,T)$ with values in 
${\mathbb{L}}_\sigma^n(\Omega)$ and again thanks Corollary~\ref{cor:StokesImprovedRegularity} with $p=\frac{2n}{3}$, 
$\alpha=\frac{1}{2}$ and $q=n$ we have the following estimate for all $t\in (0,T)$
\[
\bigl\|-{\mathscr{S}}_n*{\mathbb{P}}_{\frac{2n}{3}}\bigl((u_1\cdot\nabla)u_2\bigr)(t)\bigr\|_n
\lesssim \Bigl(\int_0^t\frac{1}{(t-s)^{\frac{1}{4}}}\frac{1}{s^{\frac{3}{4}}}\, {\rm d}s\Bigr) 
\bigl\|u_1\bigr\|_{{\mathscr{U}}_T}\bigl\|u_2\bigr\|_{{\mathscr{U}}_T}
\lesssim \bigl\|u_1\bigr\|_{{\mathscr{U}}_T}\bigl\|u_2\bigr\|_{{\mathscr{U}}_T}.
\]
This proves the continuity of $B_1(u_1,u_2)$ on $[0,T)$ with values in ${\mathbb{L}}_\sigma^n(\Omega)$ for all
$u_1,u_2 \in {\mathscr{U}}_T$.
\end{proof}

\begin{proof}[Proof of Lemma~\ref{lem:B2}]
The proof of Lemma~\ref{lem:B2} follows exactly the same lines as the one of Lemma~\ref{lem:B1} once
we check that 
$s\mapsto \delta b_1(s)\lrcorner\, b_2(s) \in {\mathscr{C}}((0,T);L^{\frac{2n}{3}}(\Omega,{\mathds{R}}^n))$
with 
\[
\bigl\|s^{\frac{3}{4}}\bigl(\delta b_1(s)\lrcorner\, b_2(s)\bigr)\bigr\|_{\frac{2n}{3}}
\le \bigl\|\sqrt{s}\, \delta b_1(s)\bigr\|_n \bigl\|s^{\frac{1}{4}}b_2(s)\bigr\|_{2n}
\le \bigl\|b_1\bigr\|_{{\mathscr{B}}_T}\bigl\|b_2\bigr\|_{{\mathscr{B}}_T}
\]
which is the case thanks to the definition of ${\mathscr{B}}_T$.
\end{proof}

\begin{proof}[Proof of Lemma~\ref{lem:B3}]
As in the proof of the Lemmas~\ref{lem:B1} and \ref{lem:B2}, we have that 
$s\mapsto d( u(s)\lrcorner\, b(s)) \in {\mathscr{C}}((0,T);L^{\frac{2n}{3}}(\Omega,\Lambda^2))$
with
\[
\bigl\|s^{\frac{3}{4}}d( u(s)\lrcorner\, b(s))\bigr\|_{\frac{2n}{3}}\le 
\bigl\|u\bigr\|_{{\mathscr{U}}_T}\bigl\|b\bigr\|_{{\mathscr{B}}_T}.
\]
This is where we use our magic formula \eqref{eq:magic} or more precisely the estimate \eqref{eq:magic-est}:
since $d b=0$ for $b \in {\mathscr{B}}_T$, Corollary~\ref{cor:magic-est} gives for all $s\in (0,T)$
\[
\bigl\|s^{\frac{3}{4}}d(u\lrcorner\,b)(s)\bigr\|_{\frac{2n}{3}}\le 
C\bigl(\bigl\|\sqrt{s} \,\nabla u(s)\bigr\|_n \bigl\|s^{\frac{1}{4}}b(s)\bigr\|_{2n} 
+\bigl\|s^{\frac{1}{4}}u(s)\bigr\|_{2n}\bigl\|\sqrt{s} \,\delta\, b(s)\bigr\|_n\bigr)
\lesssim \bigl\|u\bigr\|_{{\mathscr{U}}_T}\bigl\|b\bigr\|_{{\mathscr{B}}_T}.
\]
To mimic the proof of the previous lemmas, we will use the properties of the Hodge-Laplacian stated in 
Corollary~\ref{cor:HodgeLaplacianImprovedRegularity} $(ii)$ inherited by the operator $M$. 
We have, with $p=\frac{2n}{3}$, $\alpha=1$ and $q=2n$,
\[
\sup_{t>0}\bigl\| t^{\frac{1}{2}}{\mathscr{M}}_{\frac{2n}{3}}(t){\mathbb{Q}}_{\frac{2n}{3}}
\bigr\|_{L^{\frac{2n}{3}}(\Omega,\Lambda^2)\to L^{2n}(\Omega,\Lambda^2)} < \infty
\] 
and with $p=\frac{2n}{3}$, $\alpha=\frac{1}{2}$ and $q=n$,
\[
\sup_{t>0}\bigl\|t^{\frac{3}{4}}d {\mathscr{M}}_{\frac{2n}{3}}(t){\mathbb{Q}}_{\frac{2n}{3}}
\bigr\|_{L^{\frac{2n}{3}}(\Omega,\Lambda^2)\to L^n(\Omega,\Lambda^3)} < \infty
\]
so that for all $t\in(0,t)$,
\[
\bigl\|{\mathscr{M}}_{\frac{2n}{3}}*{\mathbb{Q}}_{\frac{2n}{3}}\bigl(d(u\lrcorner\,b)\bigr)(t)\bigr\|_{2n}
\lesssim \Bigl(\int_0^t (t-s)^{-\frac{1}{2}}s^{-\frac{3}{4}}{\rm d}s\Bigr) 
\bigl\|u\bigr\|_{{\mathscr{U}}_T}\bigl\|b\bigr\|_{{\mathscr{B}}_T}
\lesssim t^{-\frac{1}{4}}\bigl\|u\bigr\|_{{\mathscr{U}}_T}\bigl\|b\bigr\|_{{\mathscr{B}}_T}
\]
and
\[
\bigl\|\delta {\mathscr{M}}_{\frac{2n}{3}}*{\mathbb{Q}}_{\frac{2n}{3}}\bigl(d(u\lrcorner\,b)\bigr)(t)\bigr\|_{n}
\lesssim \Bigl(\int_0^t (t-s)^{-\frac{3}{4}}s^{-\frac{3}{4}}{\rm d}s\Bigr) 
\bigl\|u\bigr\|_{{\mathscr{U}}_T}\bigl\|b\bigr\|_{{\mathscr{B}}_T}
\lesssim t^{-\frac{1}{2}}\bigl\|u\bigr\|_{{\mathscr{U}}_T}\bigl\|b\bigr\|_{{\mathscr{B}}_T}.
\]
This proves that $B_3(u,b)\in {\mathscr{B}}_T$ for all $u\in{\mathscr{U}}_T$ and $b\in {\mathscr{B}}_T$. Note that
the constant in the estimate $\|B_3(u,b)\|_{{\mathscr{B}}_T}\lesssim \|u\|_{{\mathscr{U}}_T}\|b\|_{{\mathscr{B}}_T}$ is independent 
of $T>0$.
The continuity of $B_3(u,b)$ on $[0,T)$ with values in ${\sf R}_n(\underline{d})_{|_{\Lambda^2}}$ follows from
 the continuity of $d( u\lrcorner\, b)$ on $(0,T)$ with values in
$L^{\frac{2n}{3}}(\Omega,\Lambda^2))$ and the strong continuity of the semigroup ${\mathscr{M}}$: the convolution
${\mathscr{M}}_{\frac{2n}{3}}*{\mathbb{Q}}_{\frac{2n}{3}}\bigl(d(u\lrcorner\,b)\bigr)$ is continuous on $(0,T)$ with values in 
${\mathbb{L}}_\sigma^n(\Omega)$ and we have the following estimate for all $t\in (0,T)$
\[
\bigl\|{\mathscr{M}}_{\frac{2n}{3}}*{\mathbb{Q}}_{\frac{2n}{3}}\bigl(d(u\lrcorner\,b)\bigr)(t)\bigr\|_n
\lesssim \Bigl(\int_0^t\frac{1}{(t-s)^{\frac{1}{4}}}\frac{1}{s^{\frac{3}{4}}}\, {\rm d}s\Bigr) 
\bigl\|u\bigr\|_{{\mathscr{U}}_T}\bigl\|b\bigr\|_{{\mathscr{B}}_T}
\lesssim \bigl\|u\bigr\|_{{\mathscr{U}}_T}\bigl\|b\bigr\|_{{\mathscr{B}}_T},
\]
the last constant in the estimate being independent of $t\in [0,T]$.
This proves the continuity of $B_3(u,b)$ on $[0,T)$ with values in ${\sf R}_n(\underline{d})_{|_{\Lambda^2}}$ for all
$u\in {\mathscr{U}}_T$, $b\in {\mathscr{B}}_T$.
\end{proof}


\end{document}